\documentclass{amsart}

\usepackage[shortlabels]{enumitem}
\usepackage{xcolor}
\usepackage{stackengine}
\usepackage{dsfont}
\usepackage{enumitem}
\usepackage{physics}
\usepackage{amssymb}
\usepackage[T1]{fontenc}
\usepackage{amsmath}
\newtheorem{thm}{Theorem}[section]

\newtheorem{lem}[thm]{Lemma}

\newtheorem{Defi}[thm]{Definition}
\newtheorem{rk}[thm]{Remark}
\usepackage{mathtools}
\DeclarePairedDelimiter{\Vector}{\lparen}{\rparen}

\theoremstyle{definition}

\theoremstyle{remark}

\numberwithin{equation}{section}

\usepackage[babel,german=quotes]{csquotes}
\usepackage[backend=bibtex8,maxbibnames=99]{biblatex}
\bibliography{Texfile}

\begin{document}


\title[On the vanishing viscosity limit for 3D axisymmetric flows]{On the vanishing viscosity limit for 3D axisymmetric flows without swirl}


\author{Patrick Brkic}
\address{Institut für Angewandte Analysis, Universit\"at Ulm}
\email{patrick.brkic@uni-ulm.de}

\author{Emil Wiedemann}
\address{Department of Mathematics, Friedrich-Alexander-Universit\"at Erlangen-N\"urnberg}
\email{emil.wiedemann@fau.de}

\begin{abstract}
We study the vanishing viscosity limit for the three-dimensional incompressible Navier-Stokes equations in terms of the relative vorticity in the setting of axisymmetric velocity fields without swirl. We show that the weak convergence of relative vorticity to a renormalized solution of the Euler equations, established by Nobili and Seis, can be upgraded to strong convergence.
\end{abstract}

\maketitle

\section{Introduction}
The evolution of an incompressible viscous flow can be described by the $3D$ Navier-Stokes equations 
\begin{align}\label{forced NavierStokes 3D}
\begin{cases}
\partial_t u^{\nu}+u^{\nu}\cdot \nabla_x u^{\nu}+ \nabla_x p^{\nu}=\nu\Delta u^\nu+f^{\nu}\quad \text{ in } (0,T)\times \mathbb{R}^3,\\
\nabla_x\cdot u^{\nu}=0,\\
u^{\nu}(0,\cdot)=u^{\nu}_0,
\end{cases}
\end{align}
where for given viscosity $\nu >0$, at a point $(t,x)$ in space-time $u^{\nu}=u^{\nu}(t,x)\in \mathbb{R}^3$ represents the velocity field, $p^{\nu}=p^{\nu}(t,x)\in\mathbb{R}$ is the pressure, $f^{\nu}=f^{\nu}(t,x)\in \mathbb{R}^3$ is an external force, and $u^{\nu}_0(x)$ is the initial velocity. Since the pioneering work of Leray \cite{Leray} in 1934, in which he proved global existence for a class of weak solutions for (\ref{forced NavierStokes 3D}), uniqueness in this Leray-Hopf class has remained an open issue for general external forces, although very recently non-uniqueness has been demonstrated for a specific force~\cite{Albritton} (see also~\cite{Lady_enskaja_1969, Buckmaster, Guillod, Jia1, Jia2} for indications of non-uniqueness of Leray-Hopf solutions).

In the presence of certain symmetries, the situation is much better. A particularly important case is that of {\em axisymmetry without swirl}, described in detail in Section~\ref{axisec} below. The idea is that this symmetry reduces the number of degrees of freedom in the domain to two (the radial and the vertical, but not the angular coordinate), so that the axisymmetric $3D$ equations should behave similarly as the $2D$ equations. 

Accordingly, global existence and uniqueness of weak solutions for~\eqref{forced NavierStokes 3D} were proved in the setting of axisymmetric initial data without swirl~\cite{UKHOVSKII196852} for initial velocities in $L^2$, initial vorticities in $L^2\cap L^{\infty}$, initial relative vorticities in $L^2\cap L^{\infty}$ and forces in $L^1_t H^1_x$.
Around the same time, Ladyzhenskaya gave an independent proof in \cite{Lady_enskaja_1969}. A few decades later, Leonardi et al.\ gave a refined proof \cite{Leonardi2013OnAS} for initial velocities in $H^2$ and external forces in $L^2_tH^1_x$. This result was then upgraded \cite{ABIDI2008592} to initial data in $H^{\frac{1}{2}}$ and external forces in $L^2_{loc,t}H^{\lambda}_x$ with $\lambda > \frac{1}{4}$.

In two dimensions, the theory for the Navier-Stokes equations is rather complete. For instance, weak solutions are known to exist globally, to be unique and even regular. We refer to \cite{constantin1988navier, ladyzhenskaya1963mathematical, alma991004374059705596, Temam, foias1979some, Temam1979NavierStokesET} for an overview of classical results in two dimensions. 
As mentioned, the axisymmetric $3D$ equations can be viewed as related to the $2D$ case, and indeed one can write out~\eqref{forced NavierStokes 3D} in cylindrical coordinates as to arrive at a formulation which resembles the $2D$ Navier-Stokes system. More precisely, this means that~\eqref{forced NavierStokes 3D} is studied in the half-plane $\mathbb{H}$ instead of the whole $\mathbb{R}^3$, cf.~\eqref{initial value problem for forced axisymmetric 3D Navier Stokes} below. 
On that account, the well-developed theory of the Navier-Stokes equation in two dimensions is helpful for the study of the $3D$ axisymmetric situation, which is not to say that the two-dimensional theory simply transfers in every respect. 
In particular, genuinely three-dimensional effects such as vortex-stretching appear. This is for instance the case in the {\em vorticity formulation} of~\eqref{forced NavierStokes 3D}, where the velocity field can be reconstructed from the vorticity $\omega^\nu:=\operatorname{curl}u^\nu$ by virtue of the {\em Biot-Savart law}
\begin{align*}
u^{\nu}(t,x)=\int_{\mathbb{R}^3}K(x-y)\times\omega^{\nu}(t,y)dy
\end{align*}
with $K(t,x)=\frac{1}{4\pi}\frac{x}{|x|^3}$.
The vorticity formulation of the Cauchy problem for the $3D$ axisymmetric Navier-Stokes equations without swirl, written in cylindrical coordinates, then becomes
\begin{align}\label{vorticity formulation NSE}
\begin{cases}
\partial_t \omega^{\nu}+\nabla \cdot \left(u^{\nu}\omega^{\nu}\right)= \nu \left(\Delta \omega^{\nu}+ \frac{1}{r}\partial_r \omega^{\nu}- \frac{1}{r^2}\omega^{\nu}\right)\text{ in } (0,T)\times \mathbb{H},\\
u^{\nu}=G\ast \omega^{\nu},\\
\omega^{\nu}(0,\cdot)=\omega^\nu_{0},
\end{cases}
\end{align}
where the spatial differential operators are taken with respect to the cylindrical variables $(r,z)\in\mathbb{H}$ and $G$ is the kernel of the $3D$ Biot-Savart law restricted to the axisymmetric swirl-free setting. Also, by abuse of notation we wrote $\omega^\nu$ for the angular component of vorticity (the radial and vertical components are both zero). This system was recently studied in \cite{sverakgallay} with homogeneous Dirichlet boundary condition. There, the authors established global well-posedness of mild solutions for initial vorticities in $L^1$. Moreover, they studied the $3D$ Biot-Savart law in the setting of axisymmetry without swirl, which in particular exhibits similarities to the Biot-Savart law in two dimensions~\cite[Section 2]{sverakgallay}. 

In this present paper, we study the important problem of the {\em vanishing viscosity limit}. The question is whether, and in what sense, the solutions of~\eqref{vorticity formulation NSE} converge as the viscosity $\nu\searrow0$. Formally, of course, one expects the limit to be a solution to the incompressible Euler equations (to be discussed in a moment). In the axisymmetric setting without swirl, when the {\em relative vorticity} $\xi^{\nu}=\frac{\omega^{\nu}}{r}$ is initially in $L^1\cap L^p(\mathbb{R}^3)$ for some $1<p<\infty$, this is true, and our main result says that the $\xi^\nu$ will {\em strongly} converge in $L^p$, uniformly in time. The rigorous formulation of our result is given below in Theorem~\ref{main theorem}.

In the inviscid framework, the evolution of an incompressible fluid is described by the $3D$ Euler equations 
\begin{align}\label{forced 3D Euler}
\begin{cases}
\partial_t u+ u\cdot \nabla_x u +\nabla_x p=f \quad\text{ in } (0,T)\times \mathbb{R}^3,\\
\nabla_x \cdot u=0,\\
u(0,\cdot)=u_0,
\end{cases}
\end{align}
where the variables have the same meaning as before. Global well-posedness of smooth solutions is still unresolved, while on the positive side local-well posedness results are classically known~\cite[Section 2.5]{MP94}. In the axisymmetric no-swirl setting, global well-posedness was settled in \cite{UKHOVSKII196852} under the assumptions $\omega_0, \frac{\omega_0}{r}\in L^2\cap L^{\infty}$ and $f\in L^1_tH^1_x$. This result was improved \cite{Danchin_2007} to $L^{\infty}$ initial vorticities lying in the Lorentz space $L^{3,1}$ and relative vorticities in $L^{3,1}$. Moreover, for smooth initial data this was carried out in \cite{Raymond, Majda} under some additional assumptions on the initial vorticity. For a corresponding result in certain Besov spaces we refer to  \cite{Abidi}.

As mentioned, the small viscosity behaviour is a very important concern in this context. In the axisymmetric setting and for zero right-hand sides, results for the nonzero swirl case can be found for instance in \cite{HmidiZerguine,Sulaiman}.
In the case of zero swirl, the small viscosity behaviour for the axisymmetric Navier-Stokes equations was investigated in~\cite{JiuQuansen,nobili2019renormalization} within the relative vorticity formulation ($\xi^{\nu}=\frac{\omega^{\nu}}{r}$)  
\begin{align}\label{relative vorticity formulation NSE}
\begin{cases}
\partial_t \xi^{\nu}+u^{\nu}\cdot \nabla \xi^{\nu} = \nu \left(\Delta \xi^{\nu}+\frac{3}{r}\partial_r \xi^{\nu}\right) \quad\text{ in } (0,T)\times \mathbb{H}, \\
u^{\nu}=G\ast \omega^{\nu},\\
\xi^{\nu}|_{t=0}=\xi^{\nu}_0
\end{cases}
\end{align}
with homogeneous Neumann boundary condition $\partial_r\xi^{\nu}=0$ on $\partial\mathbb{H}$. 
The transport estimate 
\begin{align*}
\|\xi^{\nu}\|_{L^{\infty}((0,T),L^p(\mathbb{R}^3))}\leq \|\xi^{\nu}_0\|_{L^p(\mathbb{R}^3)}
\end{align*}
was established and used to demonstrate that in the inviscid limit the relative vorticities converge weakly$^{*}$ in $L^{\infty}_tL^p_x$, for $1<p<\infty$, to a renormalized solution of the associated axisymmetric Euler equation without swirl  
\begin{align}\label{relative vorticity formulation Euler}
\begin{cases}
\partial_t \xi + u\cdot \nabla \xi = 0 \quad\text{ in } (0,T)\times \mathbb{H},\\
u=G\ast \omega,\\
\xi|_{t=0}=\xi_0,\\
\partial_r \xi=0
\end{cases}
\end{align}
~\cite[Theorem 1.2]{nobili2019renormalization}, where renormalization is understood in the sense of DiPerna-Lions \cite{dipernalions}.

In the two dimensional framework two recent works \cite{Ciampa2021,NussenzveigLopes2021} highlighted that the established weak$^{*}$ convergence of  the vorticity can be upgraded to strong convergence in $C_tL^p_x$, for $1<p<\infty$. While in \cite{Ciampa2021} the authors established the convergence on the torus and on the whole space with some convergence rates in the torus case, in \cite{NussenzveigLopes2021} the convergence was shown on the torus with possible forcing term. The main goal of this paper is to establish strong convergence of $\xi^{\nu}$ in $C_tL^p_x$, for $1<p<\infty$, in the axisymmetric case without swirl, including a forcing term. 

Now, let us highlight the main difficulties compared to \cite{Ciampa2021} and \cite{NussenzveigLopes2021}. Due to the axisymmetry and without swirl-free property,~\eqref{relative vorticity formulation NSE} and~\eqref{relative vorticity formulation Euler} are stated in two dimensions whereas we want to show convergence in three dimensions. Technical difficulties arise in adapting the proof for the two dimensional setting in~\cite{Ciampa2021} since it is not possible to derive a bound for the velocity fields $u^{\nu}$ in $L^1+L^{\infty}(\mathbb{R}^3)$ which was crucial in the proof of Ciampa et al. To overcome this difficulty, we have to use a suitable cut-off function in a way that the scaling arising from cylindrical coordinates can be controlled properly. Moreover, in the sense of the axisymmetric setting, our result can also be seen as an extension of~\cite{Ciampa2021} to external forces. Ultimately, let us note that we cannot use directly the arguments in~\cite{NussenzveigLopes2021} to prove strong convergence, since there the compactness of the torus is exploited.\\
We conclude this introduction with some comments on the nonzero swirl case. Opposed to the zero swirl setting, the situation is different in that global-well posedness is still open for the Euler equations. However, for the Navier-Stokes equations positive answers were for instance given for small angular initial velocity in $L^3$ \cite{ZhangFang} and for smallness of $\frac{u_0^{\nu}}{\sqrt{r}}$ in $L^4$ \cite{Zhang}. Furthermore in~\cite{NeustupaPokorny}, conditions on the radial and angular components of the velocity were derived for weak solutions to become strong solutions. See also~\cite{Gallagher} for a weak-strong uniqueness result.
\section{The axisymmetric swirl-free setting}\label{axisec}
\begin{Defi}
Let $u=u(r,\theta,z)$ be a vector field in cylindrical coordinates $(r,\theta,z)$ and 
\begin{align*}
e_r=\Vector{\cos(\theta),\sin(\theta),0}, \quad e_{\theta}=\Vector{-\sin(\theta),\cos(\theta),0},\quad e_z=\Vector{0,0,1}
\end{align*}
be the unit vectors in cylindrical coordinates.
\begin{enumerate}
\item[(i)] $u$ is called {\em axisymmetric} if $u$ has cylindrical symmetry in space, i.e., $u=u(r,z)$.
\item[(ii)] $u$ is called {\em swirl-free} if its angular component vanishes, i.e., $u_{\theta}=u\cdot e_\theta=0$.
\end{enumerate}
\end{Defi}
In the first step we want to reformulate~\eqref{forced NavierStokes 3D} in the framework of axisymmetry without swirl. If we assume that the swirl component of the velocity vanishes, we can pass to cylindrical coordinates and observe that the advection term may be rewritten as
\begin{align}\label{nabla in cylindrical coordinates}
u^{\nu}\cdot \nabla_x u^{\nu}=u^{\nu}\cdot \nabla u^{\nu},
\end{align}
where $\nabla_x=\Vector{\partial_x,\partial_y, \partial_z}$ is the gradient in Cartesian variables $x,y,z$, and $\nabla =\Vector{\partial_r,\partial_z}$ is the gradient in cylindrical variables $r,z$.
Moreover, if we include cylindrical symmetry in space, the Cartesian Laplacian becomes 
\begin{align}\label{laplace in cylindrical coordinates}
\Delta_x=\partial_r^2+\frac{1}{r}\partial_r +\partial_z^2\eqqcolon \Delta +\frac{1}{r}\partial_r
\end{align}
in cylindrical coordinates. Now, for $u^{\nu}, f^{\nu}$  axisymmetric and without swirl we consider the associated initial boundary value problem to (\ref{forced NavierStokes 3D}) 
\begin{align}\label{initial value problem for forced axisymmetric 3D Navier Stokes}
\begin{cases}
\partial_t u^{\nu}+ u^{\nu}\cdot \nabla u^{\nu}+\nabla p^{\nu}=\nu (\Delta + \frac{\partial_r}{r}) u^{\nu} + f^{\nu} &\text{  in } (0,T)\times \mathbb{H},\\
u^{\nu}_r=0 &\text{ on } \partial \mathbb{H},\\
u^{\nu}(0,\cdot)=u_0^{\nu} &\text{ in } \mathbb{H},\\
\partial_r (ru_r^{\nu})+\partial_z( ru_z^{\nu})=0,
\end{cases}
\end{align}
where $\mathbb{H}=\{(r,z)\in \mathbb{R}^2:~r>0, z\in \mathbb{R}\}$ and $u^{\nu}$ is interpreted as $ u^{\nu}=(u^{\nu}_r,u^{\nu}_z)$.
As a consequence of the cylindrical symmetries, the vorticity $\omega^{\nu}$ is only toroidal and becomes $(\partial_z u_r^{\nu}-\partial_r u_z^{\nu})e_{\theta}$. In this context, a short computation reveals that the relative vorticity $\xi^{\nu}=\frac{\omega^{\nu}}{r}$ fulfils the following advection-diffusion equation:
\begin{align}\label{initial value problem for forced axisymmetric 3D Navier Stokes in relative vorticity formulation}
\begin{cases}
\partial_t \xi^{\nu}+u^{\nu}\cdot \nabla \xi^{\nu} = \nu \left(\Delta \xi^{\nu}+\frac{3}{r}\partial_r \xi^{\nu}\right) + g^{\nu} &\text{ in }(0,T)\times \mathbb{H},\\
\partial_r \xi^{\nu}= 0 &\text{ on } \partial \mathbb{H},\\
\xi^{\nu}|_{t=0}=\xi^{\nu}_0 &\text{ in } \mathbb{H},\\
u^{\nu}=G\ast \omega^{\nu}, f^{\nu}=G\ast \tilde{g}^{\nu},\\
\end{cases}
\end{align}
where $\tilde{g}^{\nu}(r,z)=r g^{\nu}(r,z)=r\operatorname{curl}f^\nu\cdot e_\theta$. Analogously, we consider the following initial boundary value problem for the Euler equations~\eqref{forced 3D Euler}:
\begin{align}\label{initial value problem for forced axisymmetric 3D Euler}
\begin{cases}
\partial_t u+ u\cdot \nabla u+\nabla p=f &\text{ in } (0,T)\times \mathbb{H},\\
u_r=0 &\text{ on } \partial \mathbb{H},\\
u(0,\cdot)=u_0 &\text{ in } \mathbb{H},\\
\partial_r(ru_r)+\partial_z (ru_z)=0.
\end{cases}
\end{align}
In the inviscid limit of~\eqref{initial value problem for forced axisymmetric 3D Navier Stokes in relative vorticity formulation}, the associated transport equation for the relative vorticity then becomes 
\begin{align}\label{initial value problem for forced axisymmetric 3D Euler in relative vorticity formulation}
\begin{cases}
\partial_t \xi+u\cdot \nabla \xi = g &\text{ in }(0,T)\times \mathbb{H}, \\
\partial_r \xi=0 &\text{ on } \partial\mathbb{H},\\
\xi^{\nu}|_{t=0}=\xi^{\nu}_0 &\text{ in } \mathbb{H},\\
u=G\ast \omega, f=G\ast \tilde{g},
\end{cases}
\end{align}
and $\tilde{g}(r,z)=r g(r,z)=r\operatorname{curl}f\cdot e_\theta$.
\begin{rk}
In our arguments it will sometimes be more convenient to work in the two-dimensional setting, i.e., in $\mathbb{H}$, and sometimes in the three-dimensional setting. By abuse of notation, we will denote a function the same no matter whether its arguments are Cartesian or cylindrical coordinates, i.e., we write for instance $F(x)=F(r,z)$.

 We will use~\eqref{nabla in cylindrical coordinates} and~\eqref{laplace in cylindrical coordinates} to switch between the two-and three-dimensional settings. Moreover, for axisymmetric functions $F=F(r,z)$ we will frequently use the integral transformation
\begin{equation*}
\int_{\mathbb{R}^3}F(x)dx=2\pi\int_{\mathbb{H}}F(r,z)rdrdz.
\end{equation*} 
\end{rk}
\begin{Defi}
Let $T>0$ and $p,q\in (1,\infty)$ be given with $\frac{1}{p}+\frac{1}{q}=1$. Let $\xi_0\in L^p_{loc}(\mathbb{R}^3)$, $g\in L^1((0,T), L^p(\mathbb{R}^3))$ and $u=G\ast \omega \in L^1((0,T),L^q_{loc}(\mathbb{R}^3)^3)$ be axisymmetric such that $\nabla_x\cdot u=0$. Then $\xi\in L^{\infty}((0,T),L^p_{loc}(\mathbb{R}^3))$ is called a {\em distributional solution} to the Euler equations (\ref{relative vorticity formulation Euler}) with initial datum $\xi_0$ and right-hand-side $g$ if $\xi$ is axisymmetric and
\begin{align*}
\int_0^T\int_{\mathbb{R}^3} \xi(\partial_t \varphi + u\cdot \nabla \varphi) dx dt + \int_{\mathbb{R}^3}\xi_0 \varphi(t=0)dx=\int_0^T\int_{\mathbb{R}^3} g \varphi dx dt
\end{align*}
for all $\varphi\in C_c^{\infty}([0,T)\times \mathbb{H})$.
\end{Defi}
\begin{rk}
The advection term is not well-defined for all $p\in (1,\infty)$. Note that $u\xi$ is locally integrable in $\mathbb{R}^3$ whenever $u\omega$ is locally integrable in $\mathbb{H}$. In accordance with the two-dimensional Sobolev embedding \cite[Proposition 2.3]{sverakgallay}
\begin{align*}
\|u\|_{L^{\frac{2p}{2-p}}(\mathbb{H})}\lesssim \|\omega\|_{L^p(\mathbb{H})},
\end{align*}
this will be the case if $p\geq \frac{4}{3}$.
For smaller exponents we introduce the notion of renormalized solutions in the sense of DiPerna-Lions \cite{dipernalions}.
\end{rk}
\begin{Defi}
Let $T>0$ and $\xi_0\in L^1(\mathbb{R}^3)$. Further let $g\in L^1((0,T),L^p(\mathbb{R}))$ and $u\in L^1((0,T),L^1_{loc}(\mathbb{R}^3)^3)$ be axisymmetric such that $\nabla_x\cdot u=0$. Then $\xi\in L^{\infty}((0,T),L^1(\mathbb{R}^3))$ is called a {\em renormalized solution} of the Euler equations~\eqref{relative vorticity formulation Euler} with initial datum $\xi_0$ and right-hand-side $g$ if $\xi$ is axisymmetric, $\xi=\xi(t,r,z)$, and
\begin{align*}
\int_0^T\int_{\mathbb{R}^3}\beta(\xi)(\partial_t \varphi+ u\cdot \nabla_x \varphi) dx dt+ \int_{\mathbb{R}^3} \beta(\xi_0) \varphi(t=0) dx= \int_0^T \int_{\mathbb{R}^3} \beta'(\xi) g\varphi  dx dt
\end{align*}
for all $\varphi\in C^\infty_c([0,T)\times \mathbb{R}^3)$ and every bounded $\beta\in C^1(\mathbb{R}^3)$ that vanishes near zero and has bounded first derivative.
\end{Defi}
\begin{lem}\label{energy estimate}
Let $\nu>0$ be a given viscosity and $1\leq p< \infty$. Let also $\xi_0^{\nu}\in L^p(\mathbb{R}^3)$ and $g^{\nu}\in L^1((0,T),L^p(\mathbb{R}^3))$. If $\xi^{\nu}$ is the solution of~\eqref{initial value problem for forced axisymmetric 3D Navier Stokes in relative vorticity formulation}, then it satisfies the energy estimate
\begin{align}\label{energy estimate for Navier Stokes in relative vorticity formulation}
{\|\xi^{\nu}(t)\|_{L^p(\mathbb{R}^3)}^p}+\nu  \int_0^t\int_{\mathbb{R}^3}|\xi^{\nu}|^{p-2}|\nabla_x \xi^{\nu}|^2 dxds\leq C(p)\left( \|\xi_0^{\nu}\|_{L^p(\mathbb{R}^3)}^p+\|g^{\nu}\|_{L^1((0,T),L^p(\mathbb{R}^3))}^p\right)
\end{align}
for all $t\in [0,T]$, where the constant $C(p)$ depends only on $p$ (but not on $\nu$). 
\end{lem}
\begin{proof}
This result can for instance be consulted from \cite[Lemma 6]{nobili2019renormalization}, where it was proved for the case $f=0$. The idea is to multiply (\ref{initial value problem for forced axisymmetric 3D Navier Stokes in relative vorticity formulation}) by ${p}|\xi^{\nu}|^{p-2}\xi^{\nu}r$ and integrate in $\mathbb{H}$. We can absorb the term including the force by means of Young's inequality. In particular we have
\begin{align*}
p\int_0^t\int_{\mathbb{H}}|\xi^{\nu}|^{p-2}\xi r g^{\nu}drdzds&\leq  \frac{p}{2\pi}\int_0^T\|\xi^{\nu}\|_{L^p(\mathbb{R}^3)}^{p-1}\|g^{\nu}\|_{L^p(\mathbb{R}^3)}dx ds\\
 &\leq \varepsilon(p-1)\|\xi^{\nu}\|^p_{L^{\infty}((0,T),L^p(\mathbb{R}^3))}+C(\varepsilon)\|g^{\nu}\|_{L^1((0,T),L^p(\mathbb{R}^3))}^p,
\end{align*}
where $\epsilon$ is chosen small enough so that $\epsilon(p-1)<1$ and thus $\varepsilon(p-1)\|\xi^{\nu}\|^p_{L^{\infty}L^p}$ can be absorbed into the left hand side.
\end{proof}

\begin{thm}\label{Theorem of Nobili and Seis with forcing term}
Let $\nu >0$ and $1<p<\infty$. Let $u^{\nu}$ be the unique solution of~\eqref{initial value problem for forced axisymmetric 3D Navier Stokes} with initial velocity $u^{\nu}_0\in L^2_{loc}(\mathbb{R}^3)$ and forces {$(f^{\nu})$} chosen such that $\xi_0^{\nu}\in L^1\cap L^p(\mathbb{R}^3)$ and $g^\nu\in L^1((0,T),L^1\cap L^p(\mathbb{R}^3))$. Assume in addition
\begin{equation*}
\xi_0^{\nu}\to \xi_0\quad \text{in $L^1\cap L^p(\mathbb{R}^3)$},\quad\quad g^\nu\to g \quad \text{in $L^1((0,T),L^1\cap L^p(\mathbb{R}^3))$.}
\end{equation*}
 Then there exist  $\xi\in L^{\infty}(0,T;L^1\cap L^p(\mathbb{R}^3))$ and $u\in C([0,T],L^2_{loc}(\mathbb{R}^3)^3)$ as well as a subsequence $(\nu_k)_{k}\to 0$ such that
\begin{align*}
\xi^{\nu_k}\mathrel{\ensurestackMath{\stackon[1pt]{\rightharpoonup}{\scriptstyle\ast}}} \xi \text{ in } L^{\infty}(0,T;L^p(\mathbb{R}^3))
\end{align*}
and
\begin{align*}
u^{\nu_k}\to u \text{ in } C([0,T],L^2_{loc}(\mathbb{R}^3)^3).
\end{align*}
Furthermore $u$ is a distributional solution of~\eqref{initial value problem for forced axisymmetric 3D Euler} and $\xi$ is a renormalized solution of~\eqref{initial value problem for forced axisymmetric 3D Navier Stokes in relative vorticity formulation}.
\end{thm}
\begin{proof}
The proof is mostly a review of the unforced case~\cite[Theorem 1]{nobili2019renormalization} and \cite[Theorem 2]{nobili2019renormalization}. The convergence can be shown as in \cite[Theorem 1]{nobili2019renormalization} with almost no changes. In the case of $f=0$ the argument for renormalization \cite[Theorem 2]{nobili2019renormalization} follows \cite{CrippaSpirito} and relies on a duality argument of DiPerna-Lions \cite{dipernalions}. Therefore the proof can be adopted if we include the forcing term in the duality formula \cite[Theorem II.6]{dipernalions}.
\end{proof}
\section{Strong convergence for Navier-Stokes equation in vorticity formulation}
With $L^p_c(\mathbb{R}^3)$ we shall denote the set of $L^p$-functions with compact (essential) support. Let us now state our main theorem:
\begin{thm}\label{main theorem}
Let $p>1$ and $\xi_0\in L^p_c(\mathbb{R}^3)$ axisymmetric. Let $R>0$ and $(\xi_0^{\nu})_{\nu}\subset L^p_c(\mathbb{R}^n)$ be axisymmetric and so that $\operatorname{supp}(\xi_0^{\nu})\subset B_R(0)$ for all $\nu >0$, and $\xi_0^{\nu}\to \xi_0$ in $L^p(\mathbb{R}^3)$.

Let $(g^{\nu})_{\nu}$ be a sequence of axisymmetric functions bounded in $L^1((0,T),L^1\cap L^{p}(\mathbb{R}^3))$ and $g\in L^1((0,T),L^1\cap L^{p}(\mathbb{R}^3))$ be axisymmetric
such that $g^{\nu}\to g$ in $L^1((0,T),L^{{p}}(\mathbb{R}^3))$.
Let $u^{\nu}$ and $\xi^{\nu}$ be the solution of (\ref{relative vorticity formulation NSE}) with initial datum $\xi_0^{\nu}$ and right hand side $g^{\nu}$.
Further let $\xi$ be a renormalized solution of (\ref{relative vorticity formulation Euler}) such that $\xi^{\nu_k} \mathrel{\ensurestackMath{\stackon[1pt]{\rightharpoonup}{\scriptstyle\ast}}} \xi$ in $L^{\infty}((0,T),L^p(\mathbb{R}^3))$ for a subsequence $(\nu_k)_{k\in\mathbb{N}}$.

Then the relative vorticities converge strongly,
\begin{align*}
\xi^{\nu_k}\to \xi \text{ in } C([0,T],L^p(\mathbb{R}^3)).
\end{align*}
\end{thm}
In order to prove this, we will first deal with an auxiliary result for the inviscid limit in the axisymmetric swirl-free framework. Recall that the $3D$ axisymmetric swirl-free transport equation and advection-diffusion equation may be reformulated in the half-plane $\mathbb{H}$ as we pointed out in the introduction. To this end, consider the initial value problem for the transport equation
\begin{align}\label{transport equation for auxiliary lemma}
\begin{cases}
\partial_t \rho + b\cdot \nabla \rho = g \quad&\text{in } (0,T)\times \mathbb{H},\\
\partial_r \rho=0 \quad&\text{on } \partial\mathbb{H},\\
\rho|_{t=0}=\rho_0,
\end{cases}
\end{align}
and for given $\nu >0$ consider the initial value problem for the associated advection-diffusion equation
\begin{align}\label{Navier Stokes for auxiliary lemma}
\begin{cases}
\partial_t \rho^{\nu}+b^{\nu}\cdot \nabla \rho^{\nu} = \nu \left(\Delta \rho^{\nu}+\frac{3}{r}\partial_r \rho^{\nu}\right) + g^{\nu} \quad&\text{in } (0,T)\times \mathbb{H},\\
\partial_r \rho^\nu=0 \quad&\text{on } \partial\mathbb{H},\\
\rho^{\nu}|_{t=0}=\rho^{\nu}_0.
\end{cases}
\end{align}
In the forthcoming lemma we want to extend \cite[Lemma 3.3]{Ciampa2021}, in which the authors established a result of convergence without external forces, to the case with external forces in the axisymmetric framework. Note that the result in \cite{Ciampa2021} does not depend on the dimension while we will focus on the axisymmetric three-dimensional case without swirl.
\begin{lem}\label{auxiliary lemma}
Let $\rho_0\in L^1(\mathbb{R}^3)\cap L^{\infty}(\mathbb{R}^3)$ and $(\rho_0^{\nu})_{\nu}$ be axisymmetric so that 
\begin{align}
\rho_0^{\nu}\to \rho_0 \text{ in } L^1(\mathbb{R}^3),\label{auxiliary lemma:convergence of rhonu in L1}\\
(\rho_0^{\nu})_\nu \text{ is bounded in } L^{\infty}.\label{auxiliary lemma:weak convergence of rhonu in Linfty}
\end{align}
Let $b=(b_r,0,b_z)$ and $g$ be axisymmetric. Further let $b$ and $g$ fulfil the following conditions:
\begin{align}
&b\in L^1((0,T),W^{1,1}_{loc}(\mathbb{R}^3)^3) \label{auxiliary lemma:first condition on b},\\
&b\in L^{\infty}((0,T),L^1+L^{\infty}(\mathbb{H}))\label{auxiliary lemma:second condition on b},\\
&\operatorname{div}(b)=0 \text{ in } \mathbb{R}^3\label{auxiliary lemma:third condition on b},\\
&g\in L^1((0,T),L^1\cap L^{\infty}(\mathbb{R}^3)).
\end{align}
Let $b^{\nu}=(b^{\nu}_r,0,b^{\nu}_z)$ and $g^{\nu}$ be axisymmetric. Let $b^{\nu}$ and $g^{\nu}$ satisfy 
\begin{align}
&\operatorname{div}(b^{\nu})=0 \text{ in } \mathbb{R}^3,\\
&b^{\nu}\to b \text{ in } L^1_{loc}((0,T)\times \mathbb{R}^3) 
\label{auxiliary lemma:local convergence of bnu},\\
&(b^{\nu})_{\nu}=(b^{\nu}_1+b^{\nu}_2)_{\nu} \text{ bounded in } L^{\infty}((0,T), (L^1\cap L^s)+L^{\infty}(\mathbb{H}))\label{auxiliary lemma:boundedness of bnu}\\
&\text{ for some } s\in (1,2)\notag\\
&(g^{\nu})_{\nu} \text{ bounded in } L^1((0,T),L^1\cap L^{\infty}(\mathbb{R}^3))\label{auxiliary lemma:boundedness of gnu},\\
&g^{\nu}\to g \text{ in } L^1((0,T),L^{{p}}(\mathbb{R}^3)) \text{ for some } 1<p<\infty\label{auxiliary lemma:convergence of gnu}.
\end{align}
%
%
Let $\rho^{\nu}\in L^{\infty}((0,T),L^1\cap L^{\infty}(\mathbb{R}^3))$ be the unique solution  of (\ref{Navier Stokes for auxiliary lemma}) and let $\rho \in L^{\infty}((0,T),L^1\cap L^{\infty}(\mathbb{R}^3))$ be the unique solution of (\ref{transport equation for auxiliary lemma}).
Then for all $1\leq q<\infty$ we have
\begin{align*}
\rho^{\nu}\to \rho \text{ in }C([0,T],L^q(\mathbb{R}^3)).
\end{align*}
\end{lem}
\begin{rk}
Under the conditions~\eqref{auxiliary lemma:convergence of rhonu in L1}-\eqref{auxiliary lemma:convergence of gnu} there indeed exist unique solutions $\rho,\rho^{\nu}\in L^{\infty}((0,T),L^1\cap L^{\infty}(\mathbb{R}^3))$ of the corresponding weak formulations in $(0,T)\times\mathbb{R}^3$ of~\eqref{transport equation for auxiliary lemma} and~\eqref{Navier Stokes for auxiliary lemma}, which are in fact renormalized. 
A proof for the case $f=0$ can be found in \cite[Theorem II.3]{dipernalions}. In the general case the arguments can be adopted with almost no changes.
\end{rk}
\begin{proof}[Proof of Lemma \ref{auxiliary lemma}] First, recall that for any $p\in [1,\infty]$  we have the following estimate thanks to Lemma \ref{energy estimate}:
\begin{align}\label{auxiliary lemma:energy estimate for Navier Stokes equation in relative vorticity formulation}
{\|\rho^{\nu}(t)\|_{L^p(\mathbb{R}^3)}^p}\leq C\left(\|\rho_0^{\nu}\|_{L^p(\mathbb{R}^3)}^p+\|g^{\nu}\|_{L^1((0,T),L^p(\mathbb{R}^3))}^p\right).
\end{align}
Due to (\ref{auxiliary lemma:convergence of rhonu in L1}), (\ref{auxiliary lemma:weak convergence of rhonu in Linfty}) and (\ref{auxiliary lemma:boundedness of gnu}), equation (\ref{auxiliary lemma:energy estimate for Navier Stokes equation in relative vorticity formulation}) implies that $(\rho^{\nu})_{\nu}$ is equibounded in $L^{\infty}((0,T),L^1\cap L^{\infty}(\mathbb{R}^3))$ and $\rho_0^{\nu} \mathrel{\ensurestackMath{\stackon[1pt]{\rightharpoonup}{\scriptstyle\ast}}}{\rho_0}\text{ in } L^{\infty}((0,T),L^p(\mathbb{R}^3))$. Moreover there exists $\tilde{\rho}\in L^{\infty}((0,T),L^1\cap L^{\infty}(\mathbb{R}^3))$ such that along a subsequence we have for any $p\in (1,\infty)$
$$\rho^{\nu} \mathrel{\ensurestackMath{\stackon[1pt]{\rightharpoonup}{\scriptstyle\ast}}}\tilde{\rho}\text{ in } L^{\infty}((0,T),L^p(\mathbb{R}^3)).$$
Thanks to the linearity one can immediately deduce that $\tilde{\rho}$ is a distributional solution  of~\eqref{transport equation for auxiliary lemma} with initial datum $\rho_0$, and hence due to uniqueness $\tilde{\rho}=\rho$. Thus, the whole sequence is converging.\\

\textbf{Step 1: Strong convergence in} $L^2((0,T)\times \mathbb{R}^3)$\\
Due to renormalization we observe
\begin{align}
&\|\rho^{\nu}(t)\|_{L^2(\mathbb{R}^3)}^2 +2\nu\int_0^t\int_{\mathbb{R}^3}|\nabla \rho^\nu(s,x)|^2dxds\leq \|\rho_0^{\nu}\|_{L^2(\mathbb{R}^3)}^2+\int_0^t \int_{\mathbb{R}^3} g^{\nu}\rho^{\nu} dx 
 ds\label{auxiliary lemma:L2 energy estimate for Navier Stokes equation in relative vorticity formulation},\\
&\|\rho(t)\|_{L^2(\mathbb{R}^3)}^2=\|\rho_0\|_{L^2(\mathbb{R}^3)}^2+\int_0^t\int_{\mathbb{R}^3} g\rho dx ds\label{inviscidl2est},
\end{align}
for almost all $t\in (0,T)$. Neglecting the non-negative viscous term, this implies
\begin{align*}
\int_{\mathbb{R}^3}|\rho^{\nu}(t)|^2-|\rho(t)|^2 dx 
&\leq \int_{\mathbb{R}^3}|\rho_0^{\nu}|^2-|\rho_0|^2 dx + \|g^{\nu}-g\|_{L^1((0,T),L^{{p}}(\mathbb{R}^3))}\|\rho^{\nu}\|_{L^{\infty}((0,T)L^{{p}'}(\mathbb{R}^3))}\\
&+\int_0^t \int_{\mathbb{R}^3}g(\rho^{\nu}-\rho) dx ds.
\end{align*}
Integrating in time once more gives
\begin{align*}
\int_0^T\int_{\mathbb{R}^3}|\rho^{\nu}(t)|^2-|\rho(t)|^2 dxdt &\leq T \int_{\mathbb{R}^3}|\rho_0^{\nu}|^2-|\rho_0|^2 dx\\
&+ T\|g^{\nu}-g\|_{L^1((0,T),L^{{p}}(\mathbb{R}^3))}\|\rho^{\nu}\|_{L^{\infty}(0,T;L^{{p}'}(\mathbb{R}^3))}\\
&+\int_0^T\int_0^t \int_{\mathbb{R}^3}g(\rho^{\nu}-\rho) dx ds dt.
\end{align*}
By Fubini's Theorem, the last integral can be simplified as 
\begin{align*}
\int_0^T\int_0^t \int_{\mathbb{R}^3}g(\rho^{\nu}-\rho) dx ds dt&=\int_0^T\int_s^T\int_{\mathbb{R}^3} g(s,x)(\rho^{\nu}(s,x)-\rho(s,x)) dx dt ds\\
&=\int_0^T\int_{\mathbb{R}^3}(T-s)g(\rho^{\nu}-\rho)dx ds.
\end{align*}
Since the function $(s,x)\mapsto (T-s)g(s,x)$ is still in $L^1(0,T;L^p(\mathbb{R}^3))$ and $\rho^{\nu} \mathrel{\ensurestackMath{\stackon[1pt]{\rightharpoonup}{\scriptstyle\ast}}}{\rho}$ in $L^\infty(0,T;L^{p'}(\mathbb{R}^3))$, we may conclude
\begin{align}\label{auxiliary lemma:lim sup inequality}
\limsup_{\nu\to 0} \left(\|\rho^{\nu}\|_{L^2((0,T),L^2(\mathbb{R}^3))}^2-\|\rho\|_{L^2((0,T),L^2(\mathbb{R}^3))}^2\right)\leq 0.
\end{align}
Further we have by weak lower semicontinuity
\begin{align*}
\liminf_{\nu\to 0} \left(\|\rho^{\nu}\|_{L^2((0,T),L^2(\mathbb{R}^3))}^2-\|\rho\|_{L^2((0,T),L^2(\mathbb{R}^3))}^2\right)\geq 0.
\end{align*}
This establishes
\begin{align*}
\|\rho^{\nu}\|_{L^2((0,T),L^2(\mathbb{R}^3))}\to \|\rho\|_{L^2(0,T;L^2(\mathbb{R}^3))}.
\end{align*}
Together with the fact that $\rho^{\nu}\rightharpoonup \rho$ in $L^2(0,T;L^2(\mathbb{R}^3))$ we accomplish $\rho^{\nu}\to \rho$ in $L^2(0,T;L^2(\mathbb{R}^3))$.

As a consequence of this convergence, we obtain that there is no `anomalous dissipation' of the $L^2$ norm, that is,
\begin{equation}\label{anomalous}
\lim_{\nu\searrow 0}\nu \int_0^T\int_{\mathbb{R}^3}|\nabla \rho^\nu(s,x)|^2dxds =0.
\end{equation} 
Indeed, this follows immediately after integrating~\eqref{auxiliary lemma:L2 energy estimate for Navier Stokes equation in relative vorticity formulation} and~\eqref{inviscidl2est} in $t$ using the convergence just established as well as the convergence of the force and initial terms. 
\\\\
\textbf{Step 2: Strong convergence in} $L^q((0,T)\times \mathbb{R}^3)$ for $1<q<\infty$\\
For $q\in (1,2)$ we use H\"older's inequality with $\frac{2}{q-1}$ and $\frac{2}{3-q}$,
\begin{align*}
\int_0^T\int_{\mathbb{R}^3}|\rho^{\nu}(t)-\rho(t)|^q dx ds&=\int_0^T\int_{\mathbb{R}^3}|\rho^{\nu}(t)-\rho(t)|^{q-1}|\rho^{\nu}(t)-\rho(t)| dx ds\\
&\leq \|\rho^{\nu}-\rho\|_{L^2(0,T;L^2(\mathbb{R}^3))}^{q-1} \|\rho^{\nu}-\rho\|_{L^{\frac{2}{3-q}}(0,T;L^{\frac{2}{3-q}}(\mathbb{R}^3))}.
\end{align*}
If $q\geq 2$ we observe
\begin{align*}
\int_0^T \int_{\mathbb{R}^3}|\rho^{\nu}(t)-\rho(t)|^q dx ds&=\int_0^T \int_{\mathbb{R}^3}|\rho^{\nu}(t)-\rho(t)|^{q-1}|\rho^{\nu}(t)-\rho(t)| dx ds\\
&\leq \|\rho^{\nu}-\rho\|_{L^{2(q-1)}(0,T;L^{2(q-1)}(\mathbb{R}^3))}^{q-1}\|\rho^{\nu}-\rho\|_{L^2(0,T;L^2(\mathbb{R}^3))}.
\end{align*}
Owing to the convergence in $L^2((0,T)\times\mathbb{R}^3)$ we deduce $\rho^{\nu}\to \rho$ in $L^q((0,T)\times \mathbb{R}^3)$ for all $q\in (1,\infty)$.\\\\
\textbf{Step 3: Convergence in} $C([0,T];L^q_w(\mathbb{R}^3))$ for $1<q<\infty$\\
For $\varphi\in C_c^{\infty}(\mathbb{H})$ we define
\begin{align*}
f_{\varphi}:t\in [0,T]\mapsto 2\pi\int_{\mathbb{H}}\rho(t,r,z)\varphi(r,z)r drdz
\end{align*}
and for $\nu>0$
\begin{align*}
f_{\varphi}^{\nu}:t\in [0,T]\mapsto 2\pi\int_{\mathbb{H}} \rho^{\nu}(t,r,z)\varphi(r,z)rdrdz.
\end{align*}
Note that $f_{\varphi}$ is continuous since solutions of~\eqref{transport equation for auxiliary lemma} with initial data in $L^q$ are known to be in $C([0,T],L^q_w(\mathbb{R}^3))$ (see for instance Step 1 in the proof of Proposition 1 in~\cite{NussenzveigLopes2021}).
Moreover $(f_{\varphi}^{\nu})_{\nu}$ and $(\partial_t f_{\varphi}^{\nu})_{\nu}$ are uniformly bounded in $[0,T]$. Indeed,
\begin{align*}
\partial_t f_{\varphi}^{\nu}(t)=&\int_{\mathbb{R}^3}\rho^{\nu}(t,x)\left(b^{\nu}(t,x)\cdot\nabla_x \varphi + \nu \Delta_x \varphi(x)\right)+g^{\nu}(t,x)\varphi(x)dx\\
&+4\pi\int_{\mathbb{H}}{\nu}\partial_r \varphi(r,z)\rho^{\nu}(t,r,z)drdz,
\end{align*}
which can be seen to be bounded by virtue of (\ref{auxiliary lemma:energy estimate for Navier Stokes equation in relative vorticity formulation}), (\ref{auxiliary lemma:boundedness of bnu}) and (\ref{auxiliary lemma:boundedness of gnu}). Note that we exploited the axisymmetry and swirl-free property of $b^{\nu}$ analogously to~\eqref{nabla in cylindrical coordinates} in order to swap $b^{\nu}\cdot \nabla \varphi$ with $b^{\nu}\cdot \nabla_x \varphi$.
Thus, in combination with the weak* convergence of $(\rho^{\nu})_{\nu}$ and~\eqref{auxiliary lemma:local convergence of bnu}, we may conclude
$\partial_t f_{\varphi}^{\nu}\to \partial_t f_{\varphi}$ in $L^1(0,T)$.
By an Arzelà-Ascoli type argument, we deduce $f_{\varphi}^{\nu}\to f_{\varphi}$ uniformly in $[0,T]$.
Due to the density of $C_c^{\infty}(\mathbb{H})$ in $L^{q'}(\mathbb{H})$, it follows that $$\rho^{\nu}\to \rho \quad\text{ in } C([0,T],L^{q}_{w}(\mathbb{R}^3)).$$
\textbf{Step 4: Convergence of $L^{q}$-norms on bounded sets} for $1<q<\infty$\\
Let $1<q<\infty$ and $\nu >0$. Let us consider  
\begin{align*}
h_{\varphi}:t\in [0,T]\mapsto 2\pi\int_{\mathbb{H}}|\rho(t,r,z)|^q\varphi(r,z)rdrdz
\end{align*}
and 
\begin{align*}
h_{\varphi}^{\nu}:t\in [0,T]\mapsto 2\pi\int_{\mathbb{H}}|\rho^{\nu}(t,r,z)|^q\varphi(r,z)r drdz
\end{align*}
for $\varphi\in C_c^{\infty}(\mathbb{H})$.
Now, by differentiating $h_{\varphi}$ and $h_{\varphi}^{\nu}$ with respect to time we may invoke~\eqref{transport equation for auxiliary lemma} and~\eqref{Navier Stokes for auxiliary lemma}. Note that $b$ and $b^{\nu}$ have zero swirl components and consequently we can again swap $b\cdot \nabla \varphi$ by $b\cdot \nabla_x \varphi$ and similarly for $b^{\nu}\cdot \nabla \rho^{\nu}$. We also use that $\rho$ is renormalized:
\begin{align*}
\partial_t h_{\varphi}&=\int_{\mathbb{R}^3} |\rho|^q b\cdot\nabla_x\varphi+ qg|\rho|^{q-2}\rho\varphi~dx,\\
\partial_t h_{\varphi}^{\nu}=&\int_{\mathbb{R}^3}  |\rho^{\nu}|^q b^{\nu}\cdot\nabla_x \varphi+\nu\left(-q(q-1)|\rho^{\nu}|^{q-2}|\nabla_x\rho^{\nu}|^2 \varphi + |\rho^{\nu}|^q\Delta_x \varphi\right)+qg^{\nu}|\rho^{\nu}|^{q-2}\rho^{\nu}\varphi dx\\
&-4\pi\nu\int_{\mathbb{H}}|\rho^{\nu}|^q \partial_r \varphi drdz.
\end{align*}
Owing to Step 2 and~\eqref{auxiliary lemma:local convergence of bnu}, \eqref{auxiliary lemma:convergence of gnu}, it immediately follows 
\begin{align*}
\int_0^T\int_{\mathbb{R}^3} |\rho^{\nu}|^q b^{\nu}\cdot\nabla_x\varphi+ g^{\nu}|\rho^{\nu}|^{q-2}\rho^{\nu}\varphi~dx dt \to \int_0^T \int_{\mathbb{R}^3} |\rho|^q b\cdot\nabla_x\varphi+ g|\rho|^{q-2}\rho\varphi~dxdt
\end{align*}
as $\nu \to 0$.
The remaining terms vanish as the viscosity tends to zero. Indeed due to~\eqref{auxiliary lemma:energy estimate for Navier Stokes equation in relative vorticity formulation} and~\eqref{anomalous}, we have for any $\varphi \in C_c^{\infty}(\mathbb{H})$, as $\nu\to0$,
\begin{align*}
&\nu\int_0^T\int_{\mathbb{R}^3} q(q-1)|\rho^{\nu}|^{q-2}|\nabla_x\rho^{\nu}|^2 |\varphi| + |\rho^{\nu}|^q|\Delta_x \varphi|~dxdt\\
&\quad+4\pi\nu\int_0^T\int_{\mathbb{H}}|\rho^{\nu}|^q |\partial_r \varphi| drdz\to 0.
\end{align*}
Therefore, $\partial_t h_{\varphi}^{\nu}\to \partial_t h_{\varphi}$ in $L^1(0,T)$. Then again by the Theorem of Arzelà-Ascoli, we conclude
\begin{align*}
\int_{\mathbb{R}^3} |\rho^{\nu}|^q \varphi dx \to \int_{\mathbb{R}^3}|\rho|^q\varphi dx 
\end{align*}
uniformly in $[0,T]$. It also follows from $\partial_t h_\phi\in L^1(0,T)$ that $\rho\in C([0,T];L^q(B_R(0)))$, as seen by a standard argument that uses an approximation of the characteristic function of $B_R(0)$ by functions in $C_c^\infty$.
\\\\
\textbf{Step 5: Local strong convergence in }$C([0,T],L^q_{loc}(\mathbb{R}^3))$, $1< q <\infty$\\
We only have to combine the results of Step 3 and Step 4. Indeed for any $(t_{\nu})_{\nu}\subset [0,T]$, $t_{\nu}\to t$, and any test function $\varphi\in C_c^\infty(\mathbb H)$, Step 3 reveals 
\begin{align*}
\int_{\mathbb{R}^3}\rho^{\nu}(t_{\nu},x)\varphi(x) dx\to \int_{\mathbb{R}^3}\rho(t,x) \varphi(x) dx
\end{align*}
as $\rho^\nu\to \rho$ in $C([0,T];L^q_w(\mathbb{R}^3))$. Likewise,
\begin{align*}
\int_{B_R(0)} |\rho^{\nu}(t_{\nu},x)|^q dx \to \int_{B_R(0)} |\rho(t,x)|^q dx 
\end{align*}
for all $R>0$ due to Step 4. As weak convergence and convergence of the norms jointly imply strong convergence, we know that $\rho^\nu(t_\nu)\to \rho(t)$ in $L^q(B_R(0))$. This now implies $\rho^{\nu}\to\rho$ in $L^q(B_R(0))$ uniformly on $[0,T]$: Indeed, if this were not the case, then there would exist $\delta>0$ and a sequence $t_\nu\to t$ of times such that 
\begin{equation*}
\|\rho^\nu(t_\nu)-\rho(t_\nu)\|_q>\delta.
\end{equation*}
Since $\rho^{\nu}(t_\nu)\to\rho(t)$ and also $\rho(t_\nu)\to\rho(t)$ in $L^q(B_R(0))$, we obtain $0\geq\delta$, a contradiction.
\\\\
\textbf{Step 6: Convergence in } $C([0,T],L^q(\mathbb{R}^3))$, $1\leq q <\infty$\\
Let us first note that
\begin{align*}
\|\rho^{\nu}(t)-\rho(t)\|_{L^q(\mathbb{R}^3)}^q\leq & \int_{B_r(0)}|\rho^{\nu}(t)-\rho(t)|^q dx\\
 &+2^{q-1}\int_{B_r^c(0)}|\rho^{\nu}(t)|^q dx+2^{q-1}\int_{B_r^c(0)}|\rho(t)|^q dx
\end{align*}
for all $r>0$.
For fixed $r>0$, the first term vanishes as $\nu\to0$, uniformly in $t$, thanks to Step 5. For the two remaining terms we proceed as in~\cite{Ciampa2021} and show that for any $\varepsilon>0$, we can find some radius $r>0$ such that 
\begin{align*}
\sup_{\nu}~\sup_t \left(\int_{B_r^c(0)}|\rho^{\nu}(t)|^q dx +\int_{B_r^c(0)}|\rho(t)|^q dx\right)<\varepsilon.
\end{align*}
For $1<R_1<\frac{1}{2}R_2$ consider a smooth cut-off function with the properties 
\begin{align*}
\psi_{R_1}^{R_2}(x)=\begin{cases}0,&0<|x|<R_1\\
1, &2R_1<|x|<R_2\\
0, &|x|>2R_2,
\end{cases}
\end{align*}
$0\leq\psi_{R_1}^{R_2}\leq 1$, and
\begin{align}\label{auxiliary lemma:estimates for derivatives of cut-off}
|\nabla_x \psi_{R_1}^{R_2}|\leq \begin{cases}\frac{C}{R_1}, \text{ if } |x|\leq 2R_1\\ \frac{D}{R_2} \text{ if } |x|> 2R_1
\end{cases},\quad|\nabla_x^2 \psi_{R_1}^{R_2}|\leq \begin{cases}\frac{C}{R_1^2}, \text{ if } |x|\leq 2R_1\\ \frac{D}{R_2^2} \text{ if } |x|> 2R_1
\end{cases}.
\end{align}
Then it follows that
\begin{align}\label{auxiliary lemma:further estimate for r derivative of cut-off}
\left|\frac{\partial_r \psi_{R_1}^{R_2}}{r}\right|\leq \left|\frac{|\nabla_x \psi_{R_1}^{R_2}|}{r}\right|\leq \begin{cases}\frac{C}{R_1^2}, \text{ if } |x|< 2R_1\\ \frac{D}{R_2^2} \text{ if } |x|\geq 2R_1.
\end{cases}
\end{align} 
We can perform a similar computation as in Lemma \ref{energy estimate} if we include the cut-off function in the formulation, i.e., we multiply (\ref{Navier Stokes for auxiliary lemma}) by ${q}|\rho^{\nu}|^{q-2}\rho^{\nu}\psi_{R_1}^{R_2}r$. As a consequence the advection term does not vanish anymore and we get 
\begin{align}\label{auxiliary lemma:energy estimate including the cut-off}
\int_{\mathbb{R}^3}|\rho^{\nu}(t)|^q\psi_{R_1}^{R_2} dx &\lesssim\int_{\mathbb{R}^3}|\rho^{\nu}_0|^q\psi_{R_1}^{R_2}dx +\int_0^T\int_{\mathbb{R}^3} |b^{\nu}| |\nabla_x \psi_{R_1}^{R_2}| |\rho^{\nu}|^q +\nu  |\rho^{\nu}|^q |\Delta_x \psi_{R_1}^{R_2}| dxdt\nonumber\\
&+\int_0^T\int_{\mathbb{H}}2\nu|\rho^{\nu}|^q r\left|\frac{\partial_r\psi_{R_1}^{R_2}}{r}\right|drdz+\int_{\mathbb{R}^3}|g^{\nu}||\rho^{\nu}|^{q-1}\psi_{R_1}^{R_2} dxdt.
\end{align}
Thanks to the a priori estimate~\eqref{auxiliary lemma:energy estimate for Navier Stokes equation in relative vorticity formulation} and the axisymmetry of $\rho^{\nu}$, we have
\begin{align}\label{auxiliary lemma:boundedness far away}
C\left(\|g^{\nu}\|_{L^1((0,T),L^q(\mathbb{R}^3))}^q\right.&\left.+\|\rho_{0}^{\nu}\|_{L^q(\mathbb{R}^3)}^q\right)\geq \|\rho^{\nu}(t)\|_{L^q(\mathbb{R}^3)}^q\\\nonumber
&\geq 2\pi\int_{-\infty}^{\infty}\int_{1}^{\infty}|\rho^{\nu}(t)|^q drdz=2\pi\|\rho^{\nu}(t)\|_{L^q(\mathbb{H}\cap \{r>1\})}^q
\end{align}
for any $q\in [1,\infty)$, and 
$\|\rho^{\nu}\|_{L^{\infty}(\mathbb{R}^3)}= \|\rho^{\nu}\|_{L^{\infty}(\mathbb{H})}.$
Then the axisymmetry of $b^{\nu}$ and the decay properties of the cut-off function (\ref{auxiliary lemma:estimates for derivatives of cut-off}), (\ref{auxiliary lemma:further estimate for r derivative of cut-off}) can be used to treat the term involving the transporting field:
\begin{align*}
\int_0^T\int_{\mathbb{R}^3} |b^{\nu}||\nabla_x \psi_{R_1}^{R_2}||\rho^{\nu}|^q dxdt &\lesssim \int_0^T\int_{\mathbb{R}^3\cap (B_{2R_1}(0)\setminus B_{R_1}(0))}  \frac{|b^{\nu}|}{R_1}|\rho^{\nu}|^q dxdt\\
&\quad\quad+\int_0^T\int_{\mathbb{R}^3\cap (B_{2R_2}(0)\setminus B_{R_2}(0))} \frac{|b^{\nu}|}{R_2}|\rho^{\nu}|^q dxdt\\
&\lesssim\int_0^T\int_{\mathbb{H}\cap \{R_1<r<2R_1\})}\frac{|b^{\nu}|}{R_1}|\rho^{\nu}|^q r drdzdt\\
&\quad\quad+\int_0^T\int_{\mathbb{H}\cap \{R_2<r<2R_2\}} \frac{|b^{\nu}|}{R_2}|\rho^{\nu}|^q r drdzdt\\
&\lesssim\int_0^T\int_{\mathbb{H}\cap \{R_1<r<2R_1\}}|b^{\nu}||\rho^{\nu}|^q drdzdt\\
&\quad\quad+\int_0^T\int_{\mathbb{H}\cap \{R_2<r<2R_2\}} |b^{\nu}||\rho^{\nu}|^q  drdzdt\\
&\lesssim \| b_1^{\nu}\|_{L^{\infty}((0,T),L^{s}(\mathbb{H}))}\|\rho^{\nu}\|_{L^{qs'}((0,T),L^{qs'}(\mathbb{H}\cap \{r>R_1\}))}^q\\
&\quad\quad+\|b_2^{\nu}\|_{L^{\infty}((0,T),L^{\infty}(\mathbb{H}))}\|\rho^{\nu}\|_{L^{q}((0,T),L^{q}(\mathbb{H}\cap \{r>R_1\}))}^q
\end{align*}
for any $s\in (1,2)$ such that $b_1^{\nu}$ is bounded in $L^{\infty}((0,T),L^s(\mathbb{H}))$, where $\frac{1}{s}+\frac{1}{s'}=1$.
%
Then, in view of (\ref{auxiliary lemma:boundedness of bnu}) and (\ref{auxiliary lemma:boundedness of gnu}) we may let $R_2\to \infty$ in (\ref{auxiliary lemma:energy estimate including the cut-off}) which leads to
\begin{align*}
\int_{B_{2R_1}^c(0)}|\rho^{\nu}(t)|^qdx &\lesssim \int_{B_{R_1}^c}|\rho^{\nu}_0|^qdx+\frac{\nu}{R_1^2}\|\rho^{\nu}\|^q_{L^{\infty}((0,T),L^q(\mathbb{R}^3))}\\
&\quad+ \| b_1^{\nu}\|_{L^{\infty}((0,T),L^{s}(\mathbb{H}))}\|\rho^{\nu}\|_{L^{qs'}((0,T),L^{qs'}(\mathbb{H}\cap \{r>R_1\}))}^q\\
&\quad+\|b_2^{\nu}\|_{L^{\infty}((0,T),L^{\infty}(\mathbb{H}))}\|\rho^{\nu}\|_{L^{q}((0,T),L^{q}(\mathbb{H}\cap \{r>R_1\}))}^q\\
&\quad+\|g^{\nu}\|_{L^1((0,T),L^q(B_{R_1}^c(0))}\|\rho^{\nu}\|^{q-1}_{L^{\infty}((0,T),L^q(B_{R_1}^c(0)))}\\
&\lesssim \|\rho_0^{\nu}\|^q_{L^q(B_{R_1}^c(0)}+\|\rho^{\nu}\|_{L^{qs'}((0,T),L^{qs'}(\mathbb{H}\cap \{r>R_1\}))}^q\\
&\quad+\|\rho^{\nu}\|_{L^{q}((0,T),L^{q}(\mathbb{H}\cap \{r>R_1\}))}^q+\frac{\nu}{R_1^2}\\
&\quad+\|g^{\nu}\|_{L^1((0,T),L^q(B_{R_1}^c(0)))}\left(\|\rho_0^{\nu}\|_{	L^q(\mathbb{R}^3)}+\|g^{\nu}\|_{L^1((0,T),L^q(\mathbb{R}^3))}\right)^{q-1}\\
&\lesssim \|\rho_0^{\nu}\|^q_{L^q(B_{R_1}^c(0))}+\|\rho^{\nu}\|_{L^{qs'}((0,T),L^{qs'}(\mathbb{H}\cap \{r>R_1\}))}^q\\
&\quad+\|\rho^{\nu}\|_{L^{q}((0,T),L^{q}(\mathbb{H}\cap \{r>R_1\}))}^q+\frac{\nu}{R_1^2}
+\|g^{\nu}\|_{L^1((0,T),L^q(B_{R_1}^c(0)))}.
\end{align*}
Now we may choose $R_1>0$ such that
\allowdisplaybreaks
\begin{align*} 
&\|\rho_0^{\nu}\|^q_{L^q(B_{R_1}^c(0))}\leq \frac{\varepsilon}{5},\\
&\|\rho^{\nu}\|_{L^{qs'}((0,T),L^{qs'}(\mathbb{H}\cap \{r>R_1\}))}^q\leq \frac{\varepsilon}{5},\\
&\|\rho^{\nu}\|_{L^{q}((0,T),L^{q}(\mathbb{H}\cap \{r>R_1\}))}^q\leq \frac{\varepsilon}{5},\\
&\|g^{\nu}\|_{L^1((0,T),L^q(B_{R_1}^c(0))}\leq \frac{\varepsilon}{5},\\
&\frac{\nu}{R_1^2}\leq \frac{\varepsilon}{5},
\end{align*}
which is possible thanks to the uniform bounds on $\rho_0^{\nu}$ and $g^{\nu}$ implied by~\eqref{auxiliary lemma:convergence of rhonu in L1}, \eqref{auxiliary lemma:weak convergence of rhonu in Linfty}, \eqref{auxiliary lemma:convergence of gnu}, and the convergence of $\rho^{\nu}$ to $\rho$ in $L^q((0,T),L^q(\mathbb{R}^3))$ for $1<q<\infty$ due to Step 2.
Since $\rho\in L^{\infty}((0,T),L^q(\mathbb{R}^3))$ we may also assume
\begin{align*}
\int_{B_{2R_1}^c(0)}|\rho(t)|^q dx \leq \varepsilon.
\end{align*}
\end{proof}
\begin{proof}[Proof of Theorem \ref{main theorem}]
$\xi^{\nu}\mathrel{\ensurestackMath{\stackon[1pt]{\rightharpoonup}{\scriptstyle\ast}}} \xi$ in $L^{\infty}((0,T),L^{\tilde{p}}(\mathbb{R}^3))$ for all $\tilde{p}\in (1,p]$.
Now, following the idea of~\cite{ConstDrivas}, we consider the two linearized problems
\begin{align}\label{relative vorticity formulation for forced NSE}
\begin{cases}
\partial_t \xi_n^{\nu} + u^{\nu}\cdot \nabla \xi_n^{\nu}=\nu \Delta \xi_n^{\nu}+ \frac{3\nu}{r}\partial_r \xi_n^{\nu}+g^{\nu}\ast \psi_n, \\
\xi_n^{\nu}(0,\cdot)=\xi_0^{\nu}*\psi_n,
\end{cases}
\end{align}
and
\begin{align}\label{relative vorticity formulation for forced Euler}
\begin{cases}
\partial_t \xi_n+ u\cdot \nabla \xi_n=g\ast \psi_n,\\
\xi_n(0,\cdot)=\xi_0*\psi_n,
\end{cases}
\end{align}
where $\psi_n$ is a standard mollifier. 
Recall that $u^{\nu}$ and $u$ are obtained from $\xi^{\nu}$ and $\xi$ through the Biot-Savart law. Note that we again used the cylindrical symmetry to give a formulation in the halfplane $\mathbb{H}$. By means of the triangle inequality, it then follows that
\begin{align}
\sup_{t\in [0,T]}\|\xi^{\nu}(t)-\xi(t)\|_{L^p(\mathbb{R}^3)}\leq &\sup_{t\in [0,T]} \| \xi^{\nu}(t)-\xi_n^{\nu}(t)\|_{L^p(\mathbb{R}^3)}
+\sup_{t\in [0,T]}\| \xi_n^{\nu}(t)-\xi_n(t)\|_{L^p(\mathbb{R}^3)}\notag\\
&+\sup_{t\in [0,T]}\| \xi_n(t)-\xi(t)\|_{L^p(\mathbb{R}^3)}.\label{estimate by triangle inequality}
\end{align}
By linearity, $\tilde{\xi}_n^{\nu}=\xi_n^{\nu}-\xi^{\nu}$ satisfies
\begin{align}\label{relative vorticity formulation for forced Navier Stokes difference}
\begin{cases}
\partial_t \tilde{\xi}_n^{\nu}+u^{\nu}\cdot \nabla \tilde{\xi}_n^{\nu}=\nu \Delta \tilde{\xi}_n^{\nu}+\frac{3\nu}{r}\tilde{\xi}_n^{\nu}+ \tilde{g}_n^{\nu},\\
\tilde{\xi}_n^{\nu}(0,\cdot)={\xi}_{n}^{\nu}(0,\cdot)-{\xi}_{0}^{\nu},
\end{cases}
\end{align}
where $\tilde{g}_n^{\nu}=g^{\nu}\ast \psi_n-g^{\nu}$. 
The $L^p$ estimate (\ref{energy estimate for Navier Stokes in relative vorticity formulation}) can be used once more in order to deduce that, for fixed $n$, the first term on the right-hand-side of~\eqref{estimate by triangle inequality} tends to zero uniformly in $\nu<\nu_0$, where $\nu_0=\nu_0(n)>0$ is sufficiently small:
\begin{align*}
\sup_{t\in [0,T]} \| \xi^{\nu}(t)-\xi_n^{\nu}(t)\|_{L^p(\mathbb{R}^3)}\leq C(p)\left(\| \xi^{\nu}_0-\xi_{n}^{\nu}(0,\cdot)\|_{L^p(\mathbb{R}^3)}+ \| g^{\nu}\ast \psi_n - g^{\nu}\|_{L^1((0,T),L^p(\mathbb{R}^3))}\right).
\end{align*}
In \cite[Lemma 1]{NussenzveigLopes2021} it was noticed that $\tilde{\xi}_n=\xi_n-\xi$ as a difference of renormalized solutions satisfies 
\begin{align}\label{relative vorticity formulation for forced Euler difference}
\begin{cases}
\partial_t \tilde{\xi}_n + u\cdot \nabla \tilde{\xi}_n = \tilde{g}_n\\
\tilde{\xi}_n(0,\cdot)=\xi_{n}(0,\cdot)-\xi_0
\end{cases}
\end{align}
in the renormalized sense, where $\tilde{g}_n=g\ast \psi_n - g$ and $\tilde{\xi}_{n,0}=\xi_{n,0}-\xi_0$. As a consequence, for fixed $n$, the third term of (\ref{estimate by triangle inequality}) converges to zero due to the estimate
\begin{align*}
\|{\xi}_n-\xi\|_{L^p(\mathbb{R}^3)}\leq C(p)\left(\|{\xi}_{n}(0,\cdot)-\xi_0\|_{L^p(\mathbb{R}^3)}+\|{g}\ast\psi_n-g\|_{L^1((0,T),L^p(\mathbb{R}^3))}\right).
\end{align*}
For the remaining term we apply Lemma \ref{auxiliary lemma}. 
Because of the regularization of the initial data we obtain $\xi_{n}^{\nu}(0,\cdot)\to \xi_{n}(0,\cdot)$ in $L^1(\mathbb{R}^3)$ and $\xi_{n}^{\nu_k}(0,\cdot)\mathrel{\ensurestackMath{\stackon[1pt]{\rightharpoonup}{\scriptstyle\ast}}}\xi_{n}(0,\cdot)$ in $L^{\infty}(\mathbb{R}^3)$ as $\nu \to 0$.
Indeed the latter weak$^{*}$ convergence is true for any fixed $n$ since
\begin{align*}
\|\xi_{n}^{\nu}(0,\cdot)\|_{L^{\infty}(\mathbb{R}^3)}=\|\xi_{0}^{\nu}\ast \psi_n\|_{L^{\infty}(\mathbb{R}^3)}\leq \|\xi_0^{\nu}\|_{L^p(\mathbb{R}^3)}\|\psi_n\|_{L^q(\mathbb{R}^3)}=C(n,p)\|\xi_0^{\nu}\|_{L^p(\mathbb{R}^3)}
\end{align*}
for $\frac{1}{p}+\frac{1}{q}=1$. Further, we decompose the axisymmetric Biot-Savart law as
\begin{align*}
u_1=(\mathds{1}_{B_1(0)}G)*\omega,\quad u_2=(\mathds{1}_{B_1(0)^c}G)*\omega,
\end{align*}
where it should be carefully noted that here, $B_1(0)$ denotes the unit ball in $\mathbb{R}^2$.

We get $u=u_1+ u_2\in L^{\infty}((0,T),L^1+L^{\infty}(\mathbb{H}))$ and $(u^{\nu})_{\nu}$ is bounded in $L^{\infty}((0,T),L^1+L^{\infty}(\mathbb{H}))$. Indeed, from~\cite[(2.11)]{sverakgallay} we know that the Biot-Savart kernel (expressed in cylindrical coordinates) satisfies $|G|\lesssim\frac{1}{|\cdot|}$. It then follows by Young's inequality that
\begin{align*}
\|u_1\|_{L^1(\mathbb{H})}+\|u_2\|_{L^{\infty}(\mathbb{H})}&\lesssim \left\|\mathds{1}_{B_1(0)}\frac{1}{|\cdot|}\ast|\omega|\right\|_{L^1(\mathbb{H})}+\left\|\mathds{1}_{B_1(0)^c}\frac{1}{|\cdot|}\ast|\omega|\right\|_{L^{\infty}(\mathbb{H})}\\
&\lesssim \|\omega\|_{L^1(\mathbb{H})}\\
&= \|\xi\|_{L^1(\mathbb{R}^3)}\lesssim \|\xi_0\|_{L^1(\mathbb{R}^3)}+\|g\|_{L^1((0,T),L^1(\mathbb{R}^3))}.\\
\end{align*}
For the same reason we have
\begin{align*}
\|u_1^{\nu}\|_{L^1(\mathbb{H})}+\|u_2^{\nu}\|_{L^{\infty}(\mathbb{H})}\leq \|\xi_0^{\nu}\|_{L^1(\mathbb{R}^3)}+\|g^{\nu}\|_{L^1((0,T),L^1(\mathbb{R}^3))}.
\end{align*}
Furthermore by Theorem \ref{Theorem of Nobili and Seis with forcing term} we have
\begin{align*}
u^{\nu}\to u \text{ in } L^2((0,T),L^2_{loc}(\mathbb{R}^3)).
\end{align*}
Eventually, by imposing (\cite[Lemma 4]{nobili2019renormalization}) it follows $u\in L^1((0,T),W^{1,1}_{loc}(\mathbb{R}^3))$.
Now, for any fixed $n$ all conditions of Lemma \ref{auxiliary lemma} are met. Hence, the second term of~\eqref{estimate by triangle inequality} becomes arbitrarily small if we choose $n$ sufficiently large and then $\nu=\nu(n)$ sufficiently small.
\end{proof}
\printbibliography
\end{document}